\documentclass[12pt, twoside]{article}
\usepackage{amsmath,amsthm,amssymb}
\usepackage{times}
\usepackage{enumerate}
\usepackage{extarrows}
\usepackage{color}
\usepackage{tabularx}
\usepackage{multirow}

\usepackage {graphics}
\usepackage{graphicx}
\usepackage{url}
\def\titlerunning#1{\gdef\titrun{#1}}
\makeatletter
\def\author#1{\gdef\autrun{\def\and{\unskip, }#1}\gdef\@author{#1}}
\def\address#1{{\def\and{\\\hspace*{18pt}}\renewcommand{\thefootnote}{}%
		\footnote {#1}}%
	\markboth{\autrun}{\titrun}}
\makeatother
\def\email#1{e-mail: #1}
\def\subjclass#1{{\renewcommand{\thefootnote}{}%
		\footnote{\emph{Mathematics Subject Classification (2010):} #1}}}
\def\keywords#1{\par\medskip
	\noindent\textbf{Keywords.} #1}

\newtheorem{theorem}{Theorem}[section]
\newtheorem{lemma}[theorem]{Lemma}
\newtheorem{definition}[theorem]{Definition}
\newtheorem{proposition}[theorem]{Proposition}
\newtheorem{remark}[theorem]{Remark}

\newtheorem{example}[theorem]{Example}

\newcommand{\R}{\mathbb{R}}

\numberwithin{equation}{section}

\newcommand{\PreserveBackslash}[1]{\let\temp=\\#1\let\\=\temp}
\newcolumntype{C}[1]{>{\PreserveBackslash\centering}p{#1}}
\newcolumntype{R}[1]{>{\PreserveBackslash\raggedleft}p{#1}}
\newcolumntype{L}[1]{>{\PreserveBackslash\raggedright}p{#1}}

\newcolumntype{I}{!{\vrule width 1pt}}
\newlength\savedwidth

\begin{document}
	
	
	\baselineskip=15pt
	
	

\titlerunning{Stability  of solutions to contact Hamilton-Jacobi equation on the circle}

\title{Stability  of solutions to contact Hamilton-Jacobi equation on the circle}

\author{Yang Xu, Jun  Yan, Kai Zhao}

\date{\today}

\maketitle

\address{Yang Xu: School of Mathematical Sciences, Fudan University, Shanghai 200433, China; \email{xuyang$\_$@fudan.edu.cn}
		\and Jun Yan: School of Mathematical Sciences, Fudan University, Shanghai 200433, China;
		\email{yanjun@fudan.edu.cn}
		\and
		Kai Zhao: School of Mathematical Sciences, Tongji University, Shanghai 200092, China;
		\email{zhaokai93@tongji.edu.cn}}


\begin{abstract}
Combing the weak KAM method for contact Hamiltonian systems and the   theory of viscosity solutions for Hamilton-Jacobi equations, we study the Lyapunov stability and instability of viscosity solutions  for evolutionary contact Hamilton-Jacobi equation  in the first part.  In the second part, we study the existence and multiplicity of time-periodic solutions. 
 
%
\end{abstract}

\tableofcontents
\section{Introduction and main results}	

\subsection{Assumptions and the motivation of this paper}
%
%
%
In this paper we consider the  contact Hamilton-Jacobi equation
	\begin{equation}\label{eq:HJe}\tag{HJ$_e$}
		\partial_t w(x,t) +  H(x,\partial_x w(x,t), w(x,t))=0, \quad (x,t)\in \mathbb{S} \times [0,+\infty).
	\end{equation}
	Here $\mathbb{S}=[0,1]$ is the unit circle, $w=w(x,t)$ is a function on $\mathbb{S} \times [0,+\infty)$ which
	represents an unknown function, $H$ is a given $C^\infty$ function on $\mathbb{S}\times\R\times\mathbb{R}$ which is the so-called contact Hamiltonian. The functions $w$ and $H$ are scalar functions,  and $\partial_tw$ and $\partial_xw$ denote the derivatives $\partial{w}/\partial{t}$ and $
	\partial{w}/\partial{x}$, respectively.  We assume throughout the following:
\begin{itemize}
\item [\bf(H1)] the Hessian $\frac{\partial^2 H}{\partial p^2} (x,p,u)$ is positive definite for each $(x,p,u)\in \mathbb{S} \times \R  \times\R$;
\item  [\bf(H2)] for each $(x,u)\in \mathbb{S} \times\R$, $H(x,p,u)$ is superlinear in $p$;
\item [\bf(H3)] there is a constant $\kappa >0 $  such that
$$
\Big| \frac{\partial H}{\partial u}(x, p,u)\Big| \leqslant \kappa ,\quad \forall (x,p,u)\in \mathbb{S} \times \R  \times \R.
$$
\end{itemize}
The study of the viscosity solution theory for Hamilton-Jacobi equations has a long history dating back to the pioneering work of Crandall and Lions \cite{CL}. Since then, much progress has been made in this field. Many important works by many authors emerge, such as \cite{Bar,I,T} and the references therein. Our methods and tools used in this paper come from \cite{WWY,WWY1,WWY2,WWY3}. In \cite{WWY}, the authors establish a variational principle for contact Hamiltonian systems under assumptions (H1)-(H3). 
 The dynamical information comes from the stationary contact Hamilton-Jacobi equation 
\begin{equation}\label{eq:HJs}\tag{HJ$_s$}
	H(x,Du(x),u(x))=0, \quad x\in \mathbb{S}.
\end{equation}
Under assumptions (H1)-(H3), the authors \cite{WWY1} introduce two semigroups of operators $\{T_t^-\}_{t\geqslant 0}$ and $\{T_t^+\}_{t\geqslant 0}$.  It  has been  proved  in \cite[Theorem 1.1]{WWY1} that the function  $(x,t)\mapsto T^-_t\varphi(x)$ is the unique viscosity solution (see\cite{CL,L}  for definitions of viscosity solutions)  of the evolutionary Hamilton-Jacobi equation
$$
	\partial_t w(x,t) +  H( x,\partial_x w(x,t), w(x,t) )=0, 
\quad \quad (x,t)\in \mathbb{S} \times [0,+\infty).
$$
satisfying 	$w(x,0)=\varphi(x)$. Here and anywhere, solutions to (HJs) and (HJe) should always be understood in the sense of  viscosity.
\medskip

This paper concerns with the  stability of  solutions   and the existence of time-periodic solutions of contact Hamilton-Jacobi equations. The novelty here is that our research object is 
$$
 \textit{  contact Hamiltonians without monotonicity assumption with respect to $u$.}
$$
In order to describe our result clearly, we first recall the  definition of stability  and some known work closely related to ours. 

\medskip
      Consider the dynamical system $(C( \mathbb{S} ,\mathbb{R}),\{T_t^-\}_{t\geqslant 0})$. The study of this dynamical system involves 
  \begin{itemize}
  	\item[-] the space of continuous functions on $\mathbb{S}  \ (\ C( \mathbb{S} ,\mathbb{R})\ )$,
  	\item[-]  time ( $t\geqslant 0$ parameterizes an irreversible continuous-time process ),
  	\item[-] the time-evolution law $( \{T_t^-\}_{t\geqslant 0} )$.
  \end{itemize}
  A  characteristic feature of dynamical theories is the emphasis on asymptotic behavior, especially in   stability of fixed points and periodic orbits. It is well known that
  \begin{itemize}
	\item[-] $\varphi\in C(\mathbb{S} ,\R)$ is a common fixed point of $\{T^-_t\}_{t\geqslant 0}$ if and only if it is a   solution of the Hamilton-Jacobi equation \eqref{eq:HJs}.
In this case, $\varphi$ is a trivial periodic point of $T^-_t$, or equivalently, it is a trivial periodic viscosity solution of \eqref{eq:HJe}.
\item[-]  $\varphi\in C(\mathbb{S} ,\R)$ is a nontrivial periodic point of $T^-_t$ if and only if the function  $(x,t)\mapsto T^-_t\varphi(x)$ is a nontrivial periodic   solution of \eqref{eq:HJe}.  More precisely,  there is $t_0>0$ such that $T^-_t\varphi=T^-_{t+t_0}\varphi$ for all $t\geqslant 0$. Note that such a $t_0 > 0$ may not  be unique. If the infimum of the set of all such $t_0$'s is not 0, we say that $\varphi$ is a nontrivial periodic point of  $T^-_t$. 
\end{itemize}

\begin{definition}(Stability of solutions)
\begin{itemize}
	\item[(a)] Solution $u$ of \eqref{eq:HJs} is said to be \textbf{Lyapunov stable}, if , for every $\epsilon>0$, there exists a constant $\delta>0 $ such that, if $\varphi\in C(\mathbb{S},\R)$ satisfying $\| \varphi- u \|_\infty<\delta $,  
	then  $\| T_t^-\varphi- u \|_\infty < \epsilon$ for each $t\geqslant 0$.
	\item[(b)] Solution  $u$ of \eqref{eq:HJs} is said to be \textbf{asymptotically stable}, if it is Lyapunov stable and there exists a constant $\delta>0$  such that   $\displaystyle \lim_{t\to +\infty} \| T_t^- \varphi-u \|_\infty=0$ for any  $\varphi\in C(\mathbb{S},\R)$ satisfying $\| \varphi- u \|_\infty<\delta $.
	\item[(c)]  Solution $u$ of \eqref{eq:HJs} is said to be \textbf{Lyapunov unstable}, if   there exists  $\epsilon_0>0$ such that for any $\delta>0$ there exists $\varphi_\delta\in   C(\mathbb{S} ,\R)$, $\| \varphi_\delta-u\|_\infty<\delta$ such that there exists a moment $t_1,t_1>0$ such that $\|T_{t_1}^- \varphi_\delta-u  \|_\infty  \geqslant \epsilon_0 $.
	\end{itemize}
\end{definition}

 We recall some known  work  closely related to ours. Under assumptions  (H1)-(H3),
\begin{itemize}
			\item  Case $\frac{\partial H}{\partial u}=0$, i.e. $H(x,p,u)$ does not depend on $u$. 
			
			There is a unique constant $c(H)$, called effective Hamiltonian \cite{LPV} or Ma\~n\'e's critical value \cite{Ma}, such that equation
			\begin{align}\label{8-991-0}
				H(x,Du(x))=c(H)
			\end{align} 
			admits viscosity solutions.  Then  for any $\varphi\in C(\mathbb{S},\R) $ , the solution $T_t^-\varphi(x)$ of 
			\begin{equation}\label{8-991}
				\partial_t u(x,t)+H(x, \partial_x u(x,t))=0 , \quad (x,t)\in \mathbb{S} \times [0,+\infty)
			\end{equation}	
			 satisfies $T_t^-\varphi+c(H)t $ goes to a viscosity solution of \eqref{8-991-0} as $t\to+\infty$ \cite{Fat-b}. Assume $c(H)=0$, then all the viscosity solutions of  \eqref{8-991-0} are \textbf{Lyapunov stable}. This means equation \eqref{8-991} has no nontrivial time-periodic  viscosity solutions.
			 
			 \medskip
			\item  Case $ \frac{\partial H}{\partial u} > 0$,  i.e. $H(x,p,u)$ is  increasing about $u$. 
			
			Equation \eqref{eq:HJs} admits the unique solution $u_0$. For any $\varphi\in C( ,\R) $,   the solution $T_t^-\varphi(x)$ of   \eqref{eq:HJe} converges to the unique   solution $u_0$ of \eqref{eq:HJs} as $t\to+\infty$. This shows the unique solution   of \eqref{eq:HJs} is  \textbf{asymptotically stable}. Moreover, if there exists $\delta>0$ with $ \, \partial H/\partial u \geqslant \delta>0 $, then
			$$
			\|T_t^-\varphi-u_0  \|\leqslant \| \varphi -u_0  \| \cdot e^{- \delta t}, \quad \forall t\in [0,+\infty). 
			$$
			This means \eqref{eq:HJe}  has no nontrivial time-periodic  viscosity solutions.
			
			\medskip
			\item  Case $\frac{\partial H}{\partial u} < 0$. i.e. $H(x,p,u)$ is  decreasing about $u$. 
		 
			It was proved in \cite{WWY3} that  there exists a $u_+$ (called  the forward weak KAM solution  of \eqref{eq:HJs}) such that
			     \begin{itemize}
			     \item $\displaystyle \min_{x\in } \{ \varphi(x) -u_+ (x)\}>0 \ $ if and only if $\ \displaystyle \lim_{t\to +\infty} T_t^- \varphi(x)=+\infty$ uniformly in $x\in $;
			     \item$\displaystyle \min_{x\in } \{ \varphi(x) -u_+ (x)\}<0\ $ if and only if $ \ \displaystyle \lim_{t\to +\infty} T_t^- \varphi(x)=-\infty$ uniformly in $x\in  $.
			     		\end{itemize}
			    This shows the viscosity solutions of \eqref{eq:HJs} are \textbf{Lyapunov unstable}. Moreover, assume there are constants $\kappa>0$, $\delta >0 $  such that
		$$
		-\kappa\leqslant\frac{\partial H}{\partial u}(x, p,u) \leqslant -\delta <0 ,\quad \forall (x,p,u)\in \mathbb{S}\times\R\times\R,
		$$
		and the following assumption holds.
\begin{itemize}
	\item[\bf(A)] There exists a solution $u_-$ of \eqref{eq:HJs} satisfying
	\begin{align}\label{A}\tag{A}
			\frac{\partial H}{\partial p}(x,p,u)\Big |_{\Lambda_{u_-}}\neq 0,
		\end{align}
\end{itemize}
 where $\Lambda_{u_-}:=\{ (x,(u_-)'(x),u_-(x)) : x \in \mathcal{D}(u_-)\}$ and $\mathcal{D}(u_-) $ denotes the set of differentiable points  of $u_-$. In \cite{WYZ2023}, the authors proved  that there exist infinitely many nontrivial  time-periodic viscosity solutions with different periods of equation \eqref{eq:HJe}.  
		\end{itemize}
\begin{center}
	{\it Can we generalize the above results  from  ``monotonicity" to ``non-monotonicity"?} 
\end{center}	
		
		This is the motivation of the present paper. 

\subsection{Main results}

In this paper, on the one hand, we focus on the stability of solutions without monotonicity assumption. We are concerned about that for a class of initial data the corresponding viscosity solutions converge to asymptotic time periodic viscosity solutions. And we focus on the rate of convergence. 
To be specific, for each viscosity solution $u$ of equation \eqref{eq:HJs}, 
we are concerned with the stability of $u$ and, if stable, the rate at which the viscosity solution $\omega_{\varphi}$ of equation \eqref{eq:HJe} satisfying $\omega_{\varphi}(x,0)=\varphi(x)$ converges to $u$ at infinite time.


Similar with \cite[Lemma 2.1, 2.2]{WYZ2023}, we prove in Appendix   that assumption \eqref{A} implies 
$u_-=:u_0$   is a $C^\infty$ solution of \eqref{eq:HJs}.
Let 
	\begin{equation}\label{eq:def-B(x)}
		B(x):= \frac{\partial H}{\partial p}(x, u'_0(x),u_0(x)),\quad x\in.
	\end{equation}
		Due to assumption \eqref{A}, one can deduce that $B(x)\in C^\infty( ,\R)$ and $B(x)\neq 0$ for all $x\in$. 

\medskip
The first main result of this paper is stated as follows.		
\begin{theorem} \label{thm1}
	Assume (H1)-(H3) and \eqref{A}. Set   
	\begin{equation}\label{eq:def-mu}
	\mu :=\frac{ \int_0^1   \frac{\partial H}{\partial u} (\tau, \,  u_0'(\tau),u_0(\tau) ) \cdot \big( B(\tau )\big)^{-1}  \ d  \tau}{\int_0^1 \big( B(\tau )\big)^{-1} d \tau },
	\end{equation}
	then there exists $\delta_0\in \R^+$ such that	
	\begin{itemize}
		\item[(1)]  If  $\mu>0$, then $u_0$ is asymptotic stable on $\Omega_{\delta_0}:=\{\varphi\in  C( ,\R): u_0-\delta_0 \leqslant \varphi\leqslant  u_0+\delta_0\}$   and 
$$
	\limsup_{t\to +\infty} \frac{\ln \| T_t^- \varphi-u_0 \|_\infty }{t} \leqslant  -\mu, \quad \forall  \varphi \in  \Omega_{\delta_0 }.
	$$
	Moreover, there exists $ \varphi_0\in \Omega_{\delta_0 }$ that makes the above inequality holds.
		 \item[(2)]   If  $\mu<0$, then $u_0$ is  Lyapunov unstable.
	\end{itemize}
%
	\end{theorem}
	
	There are some explanations on Theorem \ref{thm1}.
	
	\begin{remark} Notice that we are working on the unit circle $\mathbb{S}$. Assumption \eqref{A} guarantees that there is no fixed points of the contact Hamiltonian flow $\Phi^H_t$ generated by
			\begin{equation}\label{b6}
				\left\{
				\begin{aligned}
					\dot x&=\frac{\partial H }{\partial p}(x,p,u),\\
					\dot p &=-\frac{\partial  H }{\partial x}(x,p,u)-\frac{\partial  H }{\partial u}(x,p,u) \cdot p,  \\ 
					\dot u&=\frac{\partial  H }{\partial p}(x,p,u) \cdot p- H(x,p,u),
				\end{aligned}
				\right.
			\end{equation}
		on $\Lambda_{u_-}$.The above contact Hamiltonian system is the characteristic equations of \eqref{eq:HJe}. This is why we call \eqref{eq:HJe} and \eqref{eq:HJs} contact Hamilton-Jacobi equations. 
Under assumption \eqref{A}, in Appendix, we will prove that: (1) $\Lambda_{u_-}$ coincides with the Aubry set of contact Hamiltonian system \eqref{b6}; (2) $\Lambda_{u_-}$ is a periodic orbit of $\Phi^H_t$ with periodic $\mathcal{T}$; 
(3) equation \eqref{eq:HJs} also has a unique backward KAM solution $u_-$. Moreover, $u_-=u_+=:u_0$, and $u_0$ is of class $C^\infty$.
(4)
	Define that the Aubry set $\mathcal{A} :=\{ x(s),s\in [0, \mathcal{T})\, \}$, and  we will show that  the constant $\mu$  can also be  represented  as
	$$
		\mu =\frac{1}{\mathcal{T} } \int_0^\mathcal{T} \frac{\partial H}{\partial u} \Big( x(s),d_x u_0(x(s)) , u_0(x(s)) \Big) ds , \quad  
		\mathcal{T}=\Big|\int_0^1  \big( B(\tau )\big)^{-1} d\tau \Big| \in \R^+.
		$$
	  \end{remark}

%

In this paper, on the other hand, we focus on the existence of nontrivial  time-periodic viscosity solutions of \eqref{eq:HJe}. It was proved in \cite{WYZ2023} that there exist infinitely many  time-periodic  solutions  with different periods of equation \eqref{eq:HJe}. In this paper, we show that the  monotonicity assumption in \cite{WYZ2023} is not necessary. 

\medskip
The second main result is the following.
\begin{theorem}\label{thm2}
Assume (H1)-(H3) and \eqref{A}. If $\mu<0$, then there exist  infinitely  many  nontrivial time-periodic viscosity solutions   of equation \eqref{eq:HJs}.
\end{theorem}
\begin{remark}
 For $\mu>0$, there does not exist  nontrivial time-periodic viscosity solution  of equation \eqref{eq:HJs} in $  \Omega_{\delta_0 } $.	It is not clear that the existence of nontrivial time-periodic viscosity solution in $  \Omega_{\delta_0 } $ with $\mu=0$. 
 
\end{remark}

 The rest of this paper is organized as follows. In Section 2, we give the construction of subsolutions and supersolutions. Section 3 and Section 4 are devoted to the proof of Theorem \ref{thm1} and Theorem \ref{thm2} respectively. Section 5 gives some examples of these two theorems. And we provide some explanations of the statement of these two theorems in the last part.

\section{Construction of subsolutions and supersolutions}
We suggest that readers take a look at the Appendix before going on.

In this section, we give a construction of subsolutions and supersolutions of \eqref{eq:HJs} and \eqref{eq:HJe}, mainly for the purpose of proving Theorem \ref{thm1}.
\begin{lemma} \label{0}
Define $B(x)$ as in \eqref{eq:def-B(x)}  and assume that $u_0:=u_-$  satisfies assumption \eqref{A}. Then the function $\rho: [0,1] \to \R$ defined by
\begin{equation}\label{eq:def-rho}
	\rho(x) =  \exp \Bigg\{ \int_0^x \frac{ \mu - \frac{\partial H}{\partial u} (\tau, \,  u_0'(\tau),u_0(\tau) )}{\quad  B(\tau)    \quad }  d\tau   \Bigg\}
\end{equation}
satisfies $\rho\in C^\infty(,\R )$ and
	$\rho(0)=1=\rho(1)$ ,  $ \rho (x)>0$ for any $x \in  [0,1]$.
\end{lemma}
\begin{proof}
Due to Lemma \ref{lem00} in Appendix, $u_0:=u_-$ is of class $C^\infty$.  Since $H$ is  a $C^\infty$ function on $ \times \R \times \R$, we also have $\rho\in C^\infty(,\R )$.
It is clear that  $ \rho (s)>0$ for any $s \in  [0,1]$. 
	Note that \eqref{eq:def-mu} implies that
\begin{align*}
	\rho(1)=&\, \exp \Bigg\{ \int_0^1 \frac{ \mu - \frac{\partial H}{\partial u} (\tau, \,  u_0'(\tau),u_0(\tau) )}{\quad  B(\tau)    \quad }  d\tau   \Bigg\} \\
	=&\, \exp \Bigg\{ \mu \int_0^1 \big( B(\tau )\big)^{-1} d \tau-  \int_0^1   \frac{\partial H}{\partial u} (\tau, \,  u_0'(\tau),u_0(\tau) ) \cdot \big( B(\tau )\big)^{-1}  \ d  \tau \Bigg\}\\
	=&\, \exp(0)=1=\rho(0).
\end{align*}
\end{proof}

Here the construction of the function $\rho$ is used for the next construction of subsolutions and supersolutions.

Next, let us first construct some subsolutions and supersolutions of the stationary equation \eqref{eq:HJs}.
\begin{lemma} \label{lem2.1}
 Assume $\mu\neq 0$. Set  
	\begin{equation}\label{eq:def-u-epsilon}
	u_\epsilon(x):= u_0(x) -  \frac{\mu}{|\mu|} \cdot \epsilon  \rho(x).
	\end{equation}
	Then there exists $\epsilon_0>0$ such that
	\begin{itemize}
		\item [(1)] For $ \epsilon\in (0,\epsilon_0] $, $u_\epsilon(x)$ is a strict $C^\infty$ subsolution of \eqref{eq:HJs}, i.e. 
	$$
	 H(x, u_\epsilon'(x), u_\epsilon(x)  )<0, \quad \forall x\in \mathbb{S};
	 $$
	    \item [(2)] For $\epsilon\in [-\epsilon_0,0)$, $u_\epsilon(x) $ is a strict $C^\infty$ supersolution of \eqref{eq:HJs}, i.e. 
$$
H(x, u_\epsilon'(x), u_\epsilon(x)  )>0,  \quad \forall x\in \mathbb{S}.
$$
	\end{itemize}
	 \end{lemma}
\begin{proof}
	We shall first focus on the case $\mu>0$ , and the argument for the case $\mu<0$ is similar, so we omit its proof. For item (1),  set
	\begin{align*}
		\widehat H^\nu_u(x):= &\, \int_0^1 	 \frac{\partial H}{\partial u} \Big(x, \,  u_0'(x) -\nu \rho'(x),  u_0(x) - \tau  \cdot \nu \rho (x) \Big)\ d \tau , \\
 		\widehat H^\nu_{pp}(x):= &\,\int_0^1 s  \int_0^1 	 \frac{\partial^2 H}{\partial p^2} \Big(x, \,  u_0'(x) -  s\tau \cdot  \nu \rho'(x),  u_0(x)  \Big)\ d \tau ds , \\
 			\widehat H^\nu_{uu}(x):= &\,   \int_0^1	 \frac{\partial^2 H}{\partial u^2} \Big(x, \,  u_0'(x) -   \nu  \rho'(x),  u_0(x) -  \tau  \cdot  \nu \rho(x) \Big)\ d \tau   , \\
	\widehat  H^{\nu}_{up}(x):= &\,  \int_0^1	 \frac{\partial^2 H}{\partial u\partial p} \Big(x, \,  u_0'(x) -  \tau \cdot  \nu \rho'(x),  u_0(x) \Big)\ d \tau  ,
	\end{align*}
	and take positive constants $M_0,M_1,M_2,\epsilon_0$ as
\begin{equation}\label{eq:def-M1-M2}
	\begin{split}
			M_0:=&\, \max_{x\in \mathbb{S} , \nu \in [-1,1] } \Big\{  |	\widehat H^\nu_{pp}(x)|, |	\widehat H^\nu_{uu}(x) |,  	\widehat H^{\nu}_{up}(x)|  \Big\},\\
		M_1:=&\, \| \rho  \|_{\infty } + \| \rho'  \|_{\infty }, \quad M_2:= \min_{x\in \mathbb{S}} \rho(x) , \\
		\epsilon_0:=&\, \min\Big\{ \frac{|\mu| M_2 }{2M_0M_1(M_1+M_2)} ,1 \Big\}.
	\end{split}
\end{equation}
	 By $ H\big(x,u'_0(x),u_0(x)\big)=0,$ for any $ x\in \mathbb{S}$, we have
	\begin{equation}\label{eq:pf-lem2.1-step1}
	\begin{split}
	&\, H(x, u_\epsilon'(x), u_\epsilon(x)  )  \\
	=&\, H\Big(x,u'_0(x)-\epsilon \rho'(x),u_0(x)-\epsilon \rho(x)\Big)\\
	=&\,  H\Big(x,u'_0(x),u_0(x)\Big)+ \Big( H(x,u'_0(x)-\epsilon \rho'(x),u_0(x)-\epsilon \rho(x))  - H(x,u'_0(x)-\epsilon \rho'(x),u_0(x)) \Big) \\
	&\, \hspace{1cm} + \Big( H(x,u'_0(x)-\epsilon \rho'(x),u_0(x))- H(x,u'_0(x),u_0(x))  \Big) \\
	=&\,  H\Big(x,u'_0(x),u_0(x)\Big)+ \frac{\partial H}{\partial p} \Big(x, \,  u_0'(x),u_0(x) \Big)  \cdot (-\epsilon) \rho'(x)+\widehat H^\epsilon_{pp} \cdot(\epsilon \rho'(x) )^2 +\widehat H^\epsilon_u \cdot (-\epsilon) \rho(x)\\
	=&\,  -\epsilon \rho'(x)  B(x)+\widehat H^\epsilon _{pp} \cdot(\epsilon \rho'(x) )^2 -\epsilon \widehat H^\epsilon_u \cdot  \rho(x)\\
	=&\,  -\epsilon \mu   \rho(x) -\epsilon \rho(x)\Big(\widehat H^\epsilon_u-   \frac{\partial H}{\partial u} (x, \,  u_0'(x),u_0(x) ) \Big) + \widehat H^\epsilon_{pp} \cdot(\epsilon \rho'(x) )^2,
\end{split}
	\end{equation}
where the last equation is due to 
\begin{equation}\label{eq:rho-prime}
\begin{split}
		\rho'(x)
		=&\, \frac{\rho(x)}{B(x)}\cdot \Big( \mu - \frac{\partial H}{\partial u} (x, \,  u_0'(x),u_0(x) ) \Big ).
\end{split}
\end{equation}
Notice that
\begin{align*}
&\, \Big| \widehat H^\epsilon_u-   \frac{\partial H}{\partial u} (x, \,  u_0'(x),u_0(x) ) \Big| \\
=&\, \Big| \int_0^1 	 \frac{\partial H}{\partial u} \Big(x, \,  u_0'(x) - \epsilon \rho'(x),  u_0(x) - \theta \epsilon \rho (x) \Big)\ d \theta -   \frac{\partial H}{\partial u} (x, \,  u_0'(x),u_0(x) ) \Big| \\
\leqslant &\, \int_0^1  \Big|  \frac{\partial H}{\partial u} \Big(x, \,  u_0'(x) - \epsilon \rho'(x),  u_0(x) - \theta \epsilon \rho (x) \Big)-  \frac{\partial H}{\partial u} \Big(x, \,  u_0'(x) - \epsilon \rho'(x),  u_0(x)  \Big) \Big|  \\  &\, \quad \quad
+ \Big|  \frac{\partial H}{\partial u} \Big(x, \,  u_0'(x) - \epsilon \rho'(x),  u_0(x) \Big)- \frac{\partial H}{\partial u} (x, \,  u_0'(x),u_0(x) )    \Big|    d\theta \\
= &\,    \int_0^1   |\widehat H^{\epsilon \theta}_{uu} \cdot \epsilon \theta \rho(x) |+|\widehat H^{\epsilon}_{up} \cdot \epsilon \rho'(x)| \ d \theta \\
\leqslant &\,  |\epsilon| \cdot  M_0 M_1.
\end{align*}
Thus, for any $\epsilon\in (0,\epsilon_0] $, we have
\begin{align*}
	&\, H(x, u_\epsilon'(x), u_\epsilon(x))\\
=&\,    -\epsilon \mu   \rho(x) -\epsilon \rho(x)\Big(\widehat H^\epsilon_u-   \frac{\partial H}{\partial u} (x, \,  u_0'(x),u_0(x) ) \Big) + \widehat H^\epsilon_{pp} \cdot(\epsilon \rho'(x) )^2 \\
\leqslant &\, -\epsilon \rho(x) \cdot \Big( -\epsilon M_0M_1  +\mu   \Big) + \epsilon^2M_0 M_1^2 \\
\leqslant&\, -\epsilon M_2  \cdot \Big( -\epsilon M_0M_1 - \epsilon  \frac{M_0M_1^2}{M_2}  +\mu   \Big) \\
<&\, 0.
\end{align*}
So far, in view of Lemma \ref{0} and \eqref{eq:def-u-epsilon},  we have that $u_\epsilon(x)$ is a strict $C^\infty$ subsolution of \eqref{eq:HJs}. 

For item(2), it is quite similar with item(1). For any $\epsilon\in [-\epsilon_0,0)$, we can also get that 
\begin{align*}
	&\, H(x, u_\epsilon'(x), u_\epsilon(x))\\
=&\,    -\epsilon \mu   \rho(x) -\epsilon \rho(x)\Big(\widehat H^\epsilon_u-   \frac{\partial H}{\partial u} (x, \,  u_0'(x),u_0(x) ) \Big) + \widehat H^\epsilon_{pp} \cdot(\epsilon \rho'(x) )^2 \\
\geqslant &\, -\epsilon \rho(x) \cdot \Big( \epsilon M_0M_1  +\mu   \Big) -\epsilon^2M_0 M_1^2 \\
\geqslant&\, -\epsilon M_2  \cdot \Big( \epsilon M_0M_1 + \epsilon  \frac{M_0M_1^2}{M_2}  +\mu   \Big) \\
>&\, 0.
\end{align*}
Thus, in this case $u_\epsilon(x)$ is a strict $C^\infty$ supersolution.
\end{proof}
Based on  the construction  of subsolutions and supersolutions of the stationary equation \eqref{eq:HJs} in Lemma \ref{lem2.1}, we further construct some  subsolutions and supersolutions of the evolutionary equation \eqref{eq:HJe} with three cases in the following three lemmas.
\begin{lemma} \label{lem2.2}
 Assume $\mu> 0$. For any given $\Theta \in [0,\mu) $, set  
 \begin{equation}\label{eq:def-w-epsilon-1}
 	w_\epsilon(x,t):= u_0(x) -   \epsilon  \rho(x)e^{- \Theta t } .
 \end{equation}
	Then there exists $\widetilde \epsilon_0(\Theta)>0$ such that
	\begin{itemize}
		\item [(1)] For $ \epsilon\in (0, \widetilde \epsilon_0] $, $w_\epsilon(x,t)$ is a   $C^\infty$ subsolution of \eqref{eq:HJe}, i.e. 
	$$
	\partial_t w_\epsilon(x,t) +  H(x, w_\epsilon'(x,t), w_\epsilon(x,t) )\leqslant 0, \quad \forall (x,t)\in \mathbb{S}\times [0,+\infty).
	 $$
	    \item [(2)] For $\epsilon\in [- \widetilde \epsilon_0,0)$, $w_\epsilon(x) $ is a  $C^\infty$ supersolution of \eqref{eq:HJe}, i.e. 
	$$
	\partial_t w_\epsilon(x,t) +  H(x, w_\epsilon'(x,t), w_\epsilon(x,t) )\geqslant 0, \quad \forall (x,t)\in \mathbb{S}\times [0,+\infty).
	 $$
	\end{itemize}
	 \end{lemma}
	 \begin{proof}
	 We shall focus on the item (1) as the argument for item (2) is completely similar.
	   Similar with Lemma \ref{lem2.1}, taking
	 	\begin{equation}\label{eq:pf-lem2.2-e0}
	 		\widetilde	\epsilon_0( \Theta) := \min\Big\{ \frac{|\mu -\Theta| M_2 }{M_0M_1(M_1+M_2)}, 1\Big\},
	 	\end{equation}
	 	and combing with \eqref{eq:pf-lem2.1-step1}, one computes that
	 	\begin{align*}
	  &\, 	\partial_t w_\epsilon(x,t) +  H(x, w_\epsilon'(x,t), w_\epsilon(x,t) ) \\
	 			=&\, \Theta\cdot \epsilon e^{-\Theta t} \rho(x) - B(x)  \cdot \epsilon e^{-\Theta t}  \rho'(x)   + \widehat H^{\epsilon e^{-\Theta t} }_{pp} \cdot(\epsilon e^{-\Theta t}\rho'(x) )^2 -\widehat H^{\epsilon e^{-\Theta t} }_u \cdot  \epsilon e^{-\Theta t} \rho(x) \\
	 			=&\,    \Theta\cdot \epsilon e^{-\Theta t} \rho(x)  -\epsilon  \mu e^{-\Theta t} \rho(x) -\epsilon e^{-\Theta t} \rho(x)\Big(\widehat H^{\epsilon e^{-\Theta t}}_u-   \frac{\partial H}{\partial u} (x, \,  u_0'(x),u_0(x) ) \Big) + \widehat H^{\epsilon e^{-\Theta t}}_{pp} \cdot(\epsilon e^{-\Theta t} \rho'(x) )^2 
	 	\end{align*}
	 for any $(x,t)\in \mathbb{S} \times [0,+\infty)$. Since $e^{-\Theta t}\in (0,1]$,  we get
	 $$
	 |\widehat H^{\epsilon e^{-\Theta t} }_{pp} (x) | \leqslant M_0,   \quad  |\widehat H^{\epsilon e^{-\Theta t} }_{uu} (x) | \leqslant M_0 , \quad |\widehat H^{\epsilon e^{-\Theta t} }_{pu} (x) | \leqslant M_0 , \quad \forall (x,t)\in M\times [0,+\infty).
	 $$
	Notice that
	 $$
	 \Big| \widehat H^{\epsilon e^{-\Theta t} }_u-   \frac{\partial H}{\partial u} (x, \,  u_0'(x),u_0(x) ) \Big| \leqslant |\epsilon|e^{-\Theta t} M_0 M_1.
	 $$
	 This follows that for any $\epsilon\in (0, \widetilde \epsilon_0]$, 
\begin{align*}
	&\, \partial_t w_\epsilon (x,t) + H(x, w_\epsilon' (x,t), w_\epsilon (x,t))\\
=&\,    \Theta\cdot \epsilon e^{-\Theta t} \rho(x)  -\epsilon  \mu e^{-\Theta t} \rho(x) -\epsilon e^{-\Theta t} \rho(x)\Big(\widehat H^{\epsilon e^{-\Theta t}}_u-   \frac{\partial H}{\partial u} (x, \,  u_0'(x),u_0(x) ) \Big) \\ 
&\, \quad \quad + \widehat H^{\epsilon e^{-\Theta t}}_{pp} \cdot(\epsilon e^{-\Theta t} \rho'(x) )^2 \\
\leqslant &\, -\epsilon e^{-\Theta t} \rho(x) \cdot \Big( -\epsilon e^{-\Theta t} M_0M_1  +\mu -\Theta \Big) +(\epsilon e^{-\Theta t} )^2M_0M_1^2  \\
\leqslant &\,   -\epsilon e^{-\Theta t} \rho(x)  \cdot \Big( -\epsilon e^{-\Theta t} M_0M_1 - \epsilon e^{-\Theta t}  \frac{M_0M_1^2}{M_2}  +\mu   - \Theta  \Big) \\
\leqslant &\,  -\epsilon e^{-\Theta t} \rho(x)  \cdot \Big( -\epsilon   M_0M_1 - \epsilon    \frac{M_0M_1^2}{M_2}  +\mu   - \Theta  \Big)\\
\leqslant &\, 0.
\end{align*}
Combined with Lemma \ref{0}  and \eqref{eq:def-w-epsilon-1} , it implies that  $w_\epsilon(x,t)$ is a $C^\infty$ subsolution of \eqref{eq:HJe}.

\medskip
For item(2), it is quite similar with item(1), so we omit its proof.

\end{proof}

\begin{lemma} \label{lem2.3}
 Assume $\mu< 0$. For any given $\Theta \in (\mu,0)$, set  
	 \begin{equation}\label{eq:def-w-epsilon-2}
	w_\epsilon(x,t):= u_0(x) -   \epsilon  \rho(x)e^{- \Theta t } .
     \end{equation}
	Then there exists $\widetilde \epsilon_0(\Theta)>0 $  such that
	\begin{itemize}
		\item [(1)] For $ \epsilon\in (0, \widetilde \epsilon_0] $,  $w_\epsilon(x,t)$ is a   $C^\infty$ supersolution of \eqref{eq:HJe}, i.e. 
	$$
	\partial_t w_\epsilon(x,t) +  H(x, w_\epsilon'(x,t), w_\epsilon(x,t) )\geqslant 0, \quad \forall (x,t)\in \mathbb{S}\times \left[0, \frac{\ln 	\widetilde	\epsilon_0 -\ln |\epsilon| }{-\Theta} \right].
	 $$
	    \item [(2)] For $\epsilon\in [- \widetilde \epsilon_0,0)$,  $w_\epsilon(x) $ is a  $C^\infty$ subsolution of \eqref{eq:HJe}, i.e. 
	$$
	\partial_t w_\epsilon(x,t) +  H(x, w_\epsilon'(x,t), w_\epsilon(x,t) )\leqslant 0, \quad \forall (x,t)\in \mathbb{S}\times \left[0, \frac{\ln 	\widetilde	\epsilon_0 -\ln |\epsilon| }{-\Theta} \right].
	 $$
	\end{itemize}
	 \end{lemma}
	 \begin{proof}
      By similarity, we are only proving  item (1) .
	 	Similar with Lemma \ref{lem2.2}, by taking
	 	$$
	 		\widetilde	\epsilon_0( \Theta) :=  \min\Big\{\frac{|\Theta -\mu| M_2 }{M_0M_1(M_1+M_2)}, 1 \Big\}, \quad t_0 := \frac{\ln 	\widetilde	\epsilon_0 -\ln |\epsilon| }{-\Theta},
	 	$$
	 	it follows that $\epsilon e^{-\Theta t}\leqslant  	\widetilde	\epsilon_0 $ for any $t\in [0,t_0]$, we get that
 \begin{align*}
	&\, \partial_t w_\epsilon (x,t) + H(x, w_\epsilon' (x,t), w_\epsilon (x,t))\\
=&\,    \Theta\cdot \epsilon e^{-\Theta t} \rho(x)  -\epsilon  \mu e^{-\Theta t} \rho(x) -\epsilon e^{-\Theta t} \rho(x)\Big(\widehat H^{\epsilon e^{-\Theta t}}_u-   \frac{\partial H}{\partial u} (x, \,  u_0'(x),u_0(x) ) \Big) \\ 
&\, + \widehat H^{\epsilon e^{-\Theta t}}_{pp} \cdot(\epsilon e^{-\Theta t} \rho'(x) )^2 \\
\geqslant &\, -\epsilon e^{-\Theta t} \rho(x) \cdot \Big( \epsilon e^{-\Theta t} M_0M_1  +\mu -\Theta \Big) -(\epsilon e^{-\Theta t} )^2M_0M_1^2  \\
\geqslant &\,   -\epsilon e^{-\Theta t} \rho(x)  \cdot \Big( \epsilon e^{-\Theta t} M_0M_1 +\epsilon e^{-\Theta t}  \frac{M_0M_1^2}{M_2}  +\mu   - \Theta  \Big) \\
\geqslant &\,  -\epsilon e^{-\Theta t} \rho(x)  \cdot \Big( 	\widetilde	\epsilon_0   M_0M_1 + \	\widetilde	\epsilon_0   \frac{M_0M_1^2}{M_2}  +\mu   - \Theta  \Big)\\
\geqslant &\, 0.
\end{align*}
So far, due to Lemma \ref{0}  and \eqref{eq:def-w-epsilon-2}, we get that  $w_\epsilon(x,t)$ is a   $C^\infty$ supersolution of \eqref{eq:HJe}.
	 \end{proof}
	 
\begin{lemma}\label{lem2.4}
	 Assume $\mu> 0$. For any given $\Theta \in  (\mu,+\infty) $, set  
 \begin{equation}\label{eq:def-w-epsilon-3}
	w_\epsilon(x,t):= u_0(x) -   \epsilon  \rho(x)e^{- \Theta t } .
 \end{equation}
	Then there exists $\widetilde \epsilon_0(\Theta)>0$ such that
	\begin{itemize}
		\item [(1)] For $ \epsilon\in (0, \widetilde \epsilon_0] $, $w_\epsilon(x,t)$ is a   $C^\infty$ supersolution of \eqref{eq:HJe}, i.e. 
	$$
	\partial_t w_\epsilon(x,t) +  H(x, w_\epsilon'(x,t), w_\epsilon(x,t) )\geqslant 0, \quad \forall (x,t)\in \mathbb{S}\times [0,+\infty).
	 $$
	    \item [(2)] For $\epsilon\in [- \widetilde \epsilon_0,0)$, $w_\epsilon(x) $ is a  $C^\infty$ subsolution of \eqref{eq:HJe}, i.e. 
	$$
	\partial_t w_\epsilon(x,t) +  H(x, w_\epsilon'(x,t), w_\epsilon(x,t) )\leqslant 0, \quad \forall (x,t)\in \mathbb{S}\times [0,+\infty).
	 $$
	\end{itemize}
\end{lemma}	 
\begin{proof}

For any given $\Theta \in  (\mu,+\infty) $, similarly, we will just pay attention to item (1) since the argument for item (2) is completely similar.
	 	Similar with Lemma \ref{lem2.2}, for $ \epsilon\in (0, \widetilde \epsilon_0] $, where $\widetilde \epsilon_0$ is defined in \eqref{eq:pf-lem2.2-e0}, we have that
 \begin{align*}
	&\, \partial_t w_\epsilon (x,t) + H(x, w_\epsilon' (x,t), w_\epsilon (x,t))\\
=&\,    \Theta\cdot \epsilon e^{-\Theta t} \rho(x)  -\epsilon  \mu e^{-\Theta t} \rho(x) -\epsilon e^{-\Theta t} \rho(x)\Big(\widehat H^{\epsilon e^{-\Theta t}}_u-   \frac{\partial H}{\partial u} (x, \,  u_0'(x),u_0(x) ) \Big) \\ 
&\, + \widehat H^{\epsilon e^{-\Theta t}}_{pp} \cdot(\epsilon e^{-\Theta t} \rho'(x) )^2 \\
\geqslant &\, -\epsilon e^{-\Theta t} \rho(x) \cdot \Big( \epsilon e^{-\Theta t} M_0M_1  +\mu -\Theta \Big) -(\epsilon e^{-\Theta t} )^2M_0M_1^2  \\\geqslant &\,   -\epsilon e^{-\Theta t} \rho(x)  \cdot \Big( \epsilon e^{-\Theta t} M_0M_1 +\epsilon e^{-\Theta t}  \frac{M_0M_1^2}{M_2}  +\mu   - \Theta  \Big) \\
\geqslant &\,  -\epsilon e^{-\Theta t} \rho(x)  \cdot \Big( 	\widetilde	\epsilon_0   M_0M_1 + \	\widetilde	\epsilon_0   \frac{M_0M_1^2}{M_2}  +\mu   - \Theta  \Big)\\
\geqslant &\, 0.
\end{align*}
So far, by Lemma \ref{0}  and \eqref{eq:def-w-epsilon-3}, $w_\epsilon(x,t)$ is a   $C^\infty$ supersolution of \eqref{eq:HJe}.
\end{proof}
In the next section, we use the construction of subsolutions and supersolutions of \eqref{eq:HJe} in Lemma \ref{lem2.3} and Lemma \ref{lem2.4} to prove the Theorem \ref{thm1}.

\section{Proof of Theorem \ref{thm1}}

 Before we begin to prove the Theorem  \ref{thm1}, let us review some conclusions. 
Recall the comparison principle of Hamilton-Jacobi equation.
\begin{proposition}\label{prop:comparison principle} \cite{L}\cite{Ishii-WWY}
	Assume (H1)-(H3). For any given $T>0$, let $v,w\in C(M\times [0, T) ,\R)$  be respectively, subsolution and supersolution of 
	$$
	u_t+H(x,u_x,u)=0, \quad \forall (x,t)\in M\times (0,T).
	$$ 
	If $w(x,0)\geqslant v(x,0)$ for any $x\in M$, then $w \geqslant v$ on $ M\times [0,T)$.
\end{proposition} 
 Here $M$ is an arbitrary smooth, connected, compact Riemannian manifold without boundary.  In this paper, we take $M$ as $\mathbb{S}$.

\subsection{Proof of Theorem \ref{thm1} (1)}
For the proof of  Theorem \ref{thm1} (1), we divide it into three parts:
\begin{itemize}
	\item [{\bf Step1:}] We show that there exists $\delta_0\in \R^+$ such that	 $u_0$ is asymptotic stable on $\Omega_{\delta_0}$ for $\mu>0$.
\end{itemize}
Define  $w_{\epsilon}(x,t) $ as in \eqref{eq:def-w-epsilon-1}. Applying Lemma \ref{lem2.2}, we have that for any given $\Theta\in (0, \mu) $, there exists $\widetilde \epsilon_0(\Theta)>0$ such that $w_{\widetilde \epsilon_0}(x,t)$, $w_{-\widetilde \epsilon_0}(x,t)$ are  respectively,  $C^\infty$ subsolution and supersolution of equation \eqref{eq:HJe}.
Note that $T_t^- w_{\widetilde \epsilon_0}(x,0)$ is a solution of equation \eqref{eq:HJe}, by Proposition \ref{prop:comparison principle}, we have 
\begin{equation}\label{eq:pf-thm1-1}
\begin{split}
T_t^- w_{\widetilde \epsilon_0}(x,0) \geqslant &\, w_{\widetilde \epsilon_0}(x,t)=u_0(x) - \widetilde \epsilon_0 \rho(x)  e^{-\Theta t }, \quad \forall (x,t)\in \mathbb{S} \times [0,+\infty),\\
T_t^- w_{-\widetilde \epsilon_0}(x,0) \leqslant &\, w_{-\widetilde \epsilon_0}(x,t)=u_0(x) + \widetilde \epsilon_0 \rho(x)  e^{-\Theta t }, \quad \forall (x,t)\in \mathbb{S} \times [0,+\infty).
\end{split}
\end{equation}
For any $\varphi\in \Omega_{\widetilde \epsilon_0  M_2}$, we can get 
$$
 w_{\widetilde \epsilon_0}(x,0) =u_0- \widetilde \epsilon_0 \rho(x) \leqslant u_0-\widetilde \epsilon_0  M_2 \leqslant  \varphi \leqslant u_0+\widetilde \epsilon_0  M_2\leqslant  u_0+ \widetilde \epsilon_0 \rho(x)= w_{-\widetilde \epsilon_0}(x,0).
$$
where $M_2$ is as in \eqref{eq:def-M1-M2}.
In view of \eqref{eq:pf-thm1-1}, it implies that for any $(x,t)\in \mathbb{S} \times [0,+\infty)$,
\begin{align*}
	T_t^- \varphi \geqslant &\,   T_t^-w_{\widetilde \epsilon_0}(x,0) \geqslant u_0(x) - \widetilde \epsilon_0 \rho(x)  e^{-\Theta t },\\
		T_t^- \varphi \leqslant &\,   T_t^-w_{-\widetilde \epsilon_0}(x,0) \leqslant u_0(x) + \widetilde \epsilon_0 \rho(x)  e^{-\Theta t }.
\end{align*}
This follows that for each $\Theta\in (0, \mu)$ and  $\varphi\in \Omega_{\widetilde \epsilon_0  M_2}$, we have
\begin{equation}\label{eq:pf-thm1-2}
\| T_t^- \varphi-u_0   \|_\infty \leqslant \widetilde \epsilon_0(\Theta) M_1 e^{-\Theta t }, \quad \forall t>0.
\end{equation}
Particularly, by taking $\Theta= \mu/2 $  and 
$$
\delta_0:=  \widetilde \epsilon_0(  \mu/2  ) M_2 = \frac{ \mu  M^2_2}{2M_0 M_1 (M_1+M_2)}, 
$$
we get that $u_0$ is asymptotic stable on $\Omega_{\delta_0}$.

\begin{itemize}
	\item [{\bf Step2:}]  We show that  for any $ \varphi \in  \Omega_{\delta_0 }$,  $ \displaystyle \,
	\limsup_{t\to +\infty} t^{-1} \ln \| T_t^- \varphi-u_0 \|_\infty \leqslant  -\mu.	$
	\end{itemize}
We argue by contradiction. Suppose that there exist $\mathcal{E}\in \R^+$ and $\{t_n\}_{n\in \mathbb{N}}$  such that $\displaystyle \lim_{n\to +\infty}t_n=+\infty$ and
$$
t_n^{-1} \ln \| T_{t_n}^- \varphi-u_0 \|_\infty \geqslant -\mu +\mathcal{E}, \quad \forall n \in \mathbb{N}.
$$
This follows that
$$
\| T_{t_n}^- \varphi-u_0 \|_\infty \geqslant e^{(-\mu+\mathcal{E})t_n}, \quad \forall n\in \mathbb{N}.
$$
By taking $\Theta=\mu-\frac{1}{2}\mathcal{E}$  and  in view of \eqref{eq:pf-thm1-2}, we get
$$
   e^{(-\mu+\mathcal{E})t_n}\leqslant \| T_{t_n}^- \varphi-u_0 \|_\infty \leqslant \widetilde \epsilon_0(\Theta) M_1 e^{-\Theta t_n }.
$$
Thus,
\begin{align*}
	e^{ \frac{1}{2} \mathcal{E} t_n  } =e^{(\Theta-\mu  + \mathcal{E} ) t_n } \leqslant \widetilde \epsilon_0(\Theta) M_1=  \frac{ \mathcal{E}  M_2 }{2 M_0 (M_1+M_2)}.
\end{align*} 
It contradicts with $\displaystyle \lim_{n\to+\infty}  	e^{ \frac{1}{2} \mathcal{E} t_n  } = +\infty$.

\begin{itemize}
	\item [{\bf Step3:}]  We show that  there exists $ \varphi_0 \in  \Omega_{\delta_0 }$ such that $ \displaystyle \,
	\liminf_{t\to +\infty} \, t^{-1} \ln \| T_t^- \varphi_0-u_0 \|_\infty \geqslant  -\mu.	$
	\end{itemize}
	
Taking $\varphi_0=u_0(x) - \widetilde  \epsilon_0  \rho(x)$, we show that  $ \displaystyle \,
	\liminf_{t\to +\infty} t^{-1} \ln \| T_t^- \varphi_0-u_0 \|_\infty \geqslant  -\mu
	$. We argue by contradiction. Suppose that there exist $\mathcal{E}\in \R^+$ and $\{t_n\}_{n\in \mathbb{N}}$  such that $\displaystyle \lim_{n\to +\infty}t_n=+\infty$ and
$$
t_n^{-1} \ln \| T_{t_n}^- \varphi_0-u_0 \|_\infty \leqslant -\mu -\mathcal{E}, \quad \forall n \in \mathbb{N}.
$$
This follows that
\begin{equation}\label{eq:pf-thm1-step3-1}
\| T_{t_n}^- \varphi_0-u_0 \|_\infty \leqslant e^{(-\mu-\mathcal{E})t_n}, \quad \forall n\in \mathbb{N}.
\end{equation}
By taking $\Theta=\mu+\frac{1}{2}\mathcal{E}$  and  applying Lemma \ref{lem2.4}  and   Proposition \ref{prop:comparison principle}, we have that 
$$
 T_{t}^- \varphi_0(x) \leqslant u_0(x) - \widetilde  \epsilon_0  \rho(x) e^{- (\mu+ \frac{1}{2} \mathcal{E})   t }, \quad \forall t>0,x\in \mathbb{S}.
$$
In view of \eqref{eq:pf-thm1-step3-1}, we get  
$$
e^{ \frac{1}{2} \mathcal{E} t_n  } \leqslant ( \widetilde  \epsilon_0 M_2  )^{-1} =\frac{2 M_0 M_1 (M_2+M_1)}{ \mathcal{E}  M^2_2 }.
$$
It contradicts with $\displaystyle \lim_{n\to+\infty}  	e^{ \frac{1}{2} \mathcal{E} t_n  } = +\infty$.

 So far, we have completed the proof of Theorem \ref{thm1}(1).

\subsection{Proof of Theorem \ref{thm1} (2)}
We show that $u_0$ is  Lyapunov unstable for    $\mu<0$.

Applying Lemma \ref{lem2.3}, we get that for any fixed $\Theta\in (\mu,0) $, there exists $\widetilde \epsilon_0(\Theta)>0$ in \eqref{eq:pf-lem2.2-e0} such that for any $\epsilon \in (0, \widetilde \epsilon_0]$, $  w_{\epsilon}(x,t)$ is   a  $C^\infty$   supersolution of equation \eqref{eq:HJe} . Note that $T_t^- w_{\epsilon}(x,0)$ is a solution of equation \eqref{eq:HJe}, by Proposition \ref{prop:comparison principle}, we get that for any $\epsilon \in (0, \widetilde \epsilon_0]$,
\begin{equation}\label{eq:pf-thm1--2}
\begin{split}
T_t^- w_{\epsilon}(x,0) \leqslant &\, w_{\epsilon}(x,t)=u_0(x) - \epsilon \rho(x)  e^{-\Theta t }, \quad \forall (x,t)\in  \mathbb{S}\times \left[0, \frac{\ln 	\widetilde	\epsilon_0 -\ln |\epsilon| }{-\Theta} \right].
\end{split}
\end{equation}
For any  $\delta  \in (0,  \|\rho \|_\infty\cdot \widetilde \epsilon_0(\Theta) \, ] $,  note that
$$
 \|w_{ \delta  /\|\rho \|_\infty  }(\cdot ,0)- u_0\|_\infty   \leqslant \delta.
 $$
By \eqref{eq:pf-thm1--2} and \eqref{eq:def-M1-M2},  there exist   $\mathcal{E} := \displaystyle \min_{x\in \mathbb{S}} \rho(x)  \cdot  \widetilde \epsilon_0 = M_2 \widetilde \epsilon_0  >0$ and
$$
  t_0:=   \frac{\ln (\|\rho \|_\infty \widetilde \epsilon_0  ) - \ln ( \delta ) }{-\Theta } >0
  $$
such that  
 $$ 
 \| T_{t_0}^- w_{\delta /\|\rho \|_\infty}(\cdot ,0)-u_0\|_\infty \geqslant \frac{ \delta }{\| \rho \|_\infty } \min_{x\in \mathbb{S}} \rho(x) \cdot e^{-\Theta t_0 } = \delta M_2 \| \rho \|_\infty ^{-1} e^{-\Theta t_0}  =  \mathcal{E}.
$$
Applying Lemma \ref{lem2.1}, we obtain that $w_{\delta /\|\rho \|_\infty}(\cdot ,0)$ is a  $C^\infty$   supersolution of equation \eqref{eq:HJs}. This implies that
$$
T_t^-  w_{\delta /\|\rho \|_\infty}(x ,0)  \leqslant w_{\delta /\|\rho \|_\infty}(x,0) , \quad \forall (x,t)\in \mathbb{S} \times [0,+\infty).
$$
 Thus for any $t\geqslant t_0$, we can get
$$
 T_{t}^- w_{\delta /\|\rho \|_\infty}(\cdot ,0) =  T_{t_0}^-\mathbb{S}rc  T_{t-t_0}^- w_{\delta /\|\rho \|_\infty}(\cdot ,0) \leqslant  T_{t_0}^- w_{\delta /\|\rho \|_\infty}(\cdot ,0) \leqslant u_0-\mathcal{E}.
$$
Therefore, we have  
$$
 \| T_{t}^- w_{\delta /\|\rho \|_\infty}(\cdot ,0)-u_0\|_\infty \geqslant \mathcal{E}, \quad \forall \, t\geqslant t_0,
$$
which indicates that $u_0$ is  Lyapunov unstable for  $\mu<0$.
\qed

\section{Proof of Theorem \ref{thm2}}
We divide the proof of Theorem \ref{thm2} into several steps: from Lemma \ref{lem:thm2-1} to Lemma \ref{lem:thm2-5}. In this section we always assume (H1)-(H3) and \eqref{A}. 
\medskip

 First, we shall construct periodic subsolutions. Let 
 $$
 Z:=-\int_0^1  \big( B(\tau )\big)^{-1} d\tau \in \R \backslash \{0 \}, \quad \mathcal{T}:=|Z|=\Big|\int_0^1  \big( B(\tau )\big)^{-1} d\tau \Big| \in \R^+.
 $$
    For any fixed $x_0\in \mathbb{S}$,  we  define  
		\begin{equation}\label{eq:w}
			w (x,t):=u_0(x)+\epsilon \rho(x) +\epsilon \rho(x) \sin \Big(-\frac{\pi}{2}+  f(x)-f(x_0)+  \frac{2\pi}{Z}t\Big) ,\quad (x,t) \in \mathbb{S}\times[0,+\infty), 
		\end{equation}
where $\rho(x)$ is the function defined in \eqref{0} and of class $C^\infty$ in Lemma \ref{0} and
\begin{align*}
    & f(x) := 2 \pi \cdot \frac{\int_0^x (B(\tau))^{-1} d \tau }{\int_0^1 (B(\tau))^{-1} d \tau } =  \frac{2 \pi}{Z}  \int_0^x  - (B(\tau))^{-1} d \tau >0 , 
\end{align*}
and $\epsilon >0$ is a parameter. 

Next, we will show that $w (x,t)$ is indeed a nontrivial periodic subsolution for sufficient small $\epsilon>0$.
 \begin{lemma}\label{lem:thm2-1}
Assume $\mu <0$. There exists $\widetilde \epsilon_1>0$ such that for any $\epsilon\in (0,\widetilde \epsilon_1]$, $w(x,t)$ defined in \eqref{eq:w} is a $\mathcal{T}$-periodic subsolution of \eqref{eq:HJe}, i.e. for any $(x,t)\in \mathbb{S} \times [0,+\infty)$,
  $$
w(x,t+\mathcal{T})=w(x,t)   \quad \text{and} \quad  \partial_t \omega(x,t) + H (x, \partial_t \omega(x,t) , \omega(x,t) ) \leqslant 0.
  $$
 \end{lemma}
\begin{proof}
By \eqref{eq:w}, it is clear that 
\begin{equation}\label{eq:pf-thm2-step1-T-periodic}
w(x,t+\mathcal{T})=w(x,t), \quad \forall (x,t)\in \mathbb{S} \times [0,+\infty).
\end{equation}
 For any $x\in \mathbb{S}$, $t\geqslant 0$,   let 
$$
 F(x,t) := - \frac{\pi}{2} + f(x) - f(x_0) + \frac{2 \pi}{Z} t.
 $$ 
For $\forall \ \nu \in [0,1]$, let
\begin{align*}
	\widetilde H^\nu_u(x,t):=  &\,   \int_0^1 	 \frac{\partial H}{\partial u} \Big(x, \,  u_0'(x) +  \nu [\rho(x) \cos{F(x,t)} + \rho'(x) (1 + \sin{F(x,t)})] ,\\  &\,   u_0(x) + \nu   \tau \rho(x) (1 + \sin{F(x,t)}) )   \Big)\ d \tau, \\
    \widetilde {H}_{pp}^{\nu}(x,t) :=  &\,   \int_0^1 s \int_0^1 \frac{\partial ^2 H}{\partial p^2}  \Big(x, u_0'(x) + \nu s \tau [\rho(x) \cos{F(x,t)} + \rho'(x) (1 + \sin{F(x,t)})], u_0(x) \Big) d \tau ds , \\
     \widetilde{H}_{uu}^{\nu}(x,t) :=  &\,     \int_0^1 \frac{\partial ^2 H}{\partial u^2}  \Big(x, u_0'(x) + \nu [\rho(x) \cos{F(x,t)} + \rho'(x) (1 + \sin{F(x,t)})], \\
   & \hspace{5cm}   u_0(x)+ \tau\cdot \nu   \rho(x) (1 + \sin{F(x,t)}) \Big) d \tau  , \\
     \widetilde {H}_{pu}^{\nu}(x,t) :=  &\,   \int_0^1 \frac{\partial ^2 H}{\partial p \partial u }\Big(x,    u_0'(x) + \tau\cdot \nu [\rho(x) \cos{F(x,t)} + \rho'(x) (1 + \sin{F(x,t)})] , u_0(x) \Big)\, d\tau .
\end{align*}
It is clear that $\widetilde {H}_{pp}^{\nu}(x,t),    \widetilde{H}_{uu}^{\nu}(x,t),   \widetilde {H}_{pu}^{\nu}(x,t)$ is periodic with respect to $t$. Set 
\begin{align*}     
     \widetilde    M_0 := \Big\{ \max_{(x,t)\in \mathbb{S} \times [0, + \infty) \atop \nu \in [-1,1]} \big |& \widetilde{H}_{pp}^{\nu}(x,t) \big |  , \max_{(x,t)\in \mathbb{S} \times [0, + \infty) \atop \nu \in [-1,1]} \big | \widetilde{H}_{uu}^{\nu}(x,t) \big | ,\max_{(x,t)\in \mathbb{S} \times [0, + \infty) \atop \nu \in [-1,1]} \big | \widetilde{H}_{pu}^{\nu}(x,t) \big |  \Big\}, 
\end{align*}
and take
$$
\widetilde \epsilon_1: = \min \Big\{   \frac{-\mu}{M_1 \widetilde M_0}  \cdot \Big[   \frac{8 \pi^2}{\mathcal{T}^2 \min_{x \in \mathbb{S}}|B(x)|^2} + 2 \alpha^2 +  \frac{4 \pi \alpha }{\mathcal{T}   \min_{x \in \mathbb{S}}|B(x)|}+2+2 \alpha + \frac{2\pi}{\mathcal{T}\min_{x\in \mathbb{S}} |B(x)|}  \Big]^{-1}   , 1 \Big\}  .
$$
where $M_1$ is as in \eqref{eq:def-M1-M2}.

We assert that 	
		$w(x,t)$ defined in \eqref{eq:w} is a subsolution of \eqref{eq:HJe} with $w(x_0,0)=u_0(x_0)$.
In view of  \eqref{eq:rho-prime}, one gets
$$
\frac{\partial H}{\partial u} \Big(x,   u_0'(x), u_0(x)\Big) + \frac{\rho^{'}(x)}{\rho(x)} \cdot B(x)=\mu, \quad \forall x\in \mathbb{S},
$$
which implies 
\begin{equation}\label{eq:pf-thm1-rho-rho'}
   \Big | \frac{\rho^{'}(x)}{\rho(x)} \Big | \leqslant \alpha := \frac{\kappa + |\mu |}{\min_{x \in \mathbb{S} } | B(x)|}, \quad \forall x\in \mathbb{S} .
\end{equation}
By direct computation and $H \big(x, u_0^\prime (x) ,u_0(x) \big )=0 $,   we have 
\begin{align*}
    &\, \partial_t \omega(x,t) + H (x, \partial_x \omega(x,t) , \omega(x,t) )  \\
    = &\, \partial_t \omega(x,t) + H(x, u_0', u_0) +  \Big( H(x, \partial_x \omega, u_0) - H(x, u_0', u_0) \Big) + \Big( H(x, \partial_x \omega, \omega) - H(x, \partial_x \omega, u_0) \Big) \\ 
   \leqslant  &\,  \epsilon \rho(x) \cos{F(x,t)} \cdot \frac{ 2 \pi}{Z} + \frac{\partial H}{\partial p}( x, u_0',u_0) \cdot (\partial_x w-u_0' )+ \widetilde H_{pp}^\epsilon \cdot(\partial_x w-u_0' )^2 + \widetilde H_u^\epsilon  \cdot ( \omega -u_0). 
   \end{align*}
   Notice that 
\begin{align*}
 &\, \Big| \widetilde  H^\epsilon_u-   \frac{\partial H}{\partial u} (x, \,  u_0'(x),u_0(x) ) \Big| \\ 
=&\, \Big| \int_0^1 	 \frac{\partial H}{\partial u} \Big(x, \,  \partial_x \omega ,  u_0 + \theta (\omega-u_0 )  \Big)\ d \theta -   \frac{\partial H}{\partial u} (x, \,  u_0'(x),u_0(x) ) \Big| \\
\leqslant &\, \int_0^1 \Big|	 \frac{\partial H}{\partial u} \Big(x, \,  \partial_x \omega ,  u_0 + \theta (\omega-u_0 )  \Big)\   -   \frac{\partial H}{\partial u} (x, \,   \partial_x \omega ,u_0(x) ) \Big| d\theta \\
&\, \quad +\int_0^1 	\Big| \frac{\partial H}{\partial u} (x, \,  \partial_x \omega , u_0 (x))\   -   \frac{\partial H}{\partial u} (x, \,  u_0'(x),u_0(x) ) \Big| \ d\theta \\
= &\, \int_0^1   | (\partial_x \omega - u_0'  ) \cdot    \widetilde H^{1}_{up} |\, d \theta +\int_0^1  |(\omega-u_0)\cdot  \widetilde H^{\theta}_{uu}|  \, d \theta  \\
\leqslant &\,  | \partial_x \omega - u_0' |  \cdot   \max_{(x,t)\in \mathbb{S} \times [0, + \infty) \atop \nu \in [-1,1]} |\widetilde H^{\nu}_{up} |  + |\omega-u_0|\cdot \max_{(x,t)\in \mathbb{S} \times [0, + \infty) \atop \nu \in [-1,1]}  |\widetilde H^\nu_{uu}  |  \\
\leqslant &\,   ( | \partial_x \omega - u_0' | +   |\omega-u_0| )\cdot \widetilde M_0,
\end{align*}
and by \eqref{eq:pf-thm1-rho-rho'} , 
$$
\frac{ 2 \pi}{Z} + f'(x) \cdot B(x)=0, \quad  \frac{\partial H}{\partial u} (x,   u_0' , u_0) + \frac{\rho'(x)}{\rho(x)} \cdot B(x)=\mu,
$$
thus, we get 
\begin{align*}
	 &\, \partial_t \omega(x,t) + H (x, \partial_x \omega(x,t) , \omega(x,t) )  \\
  \leqslant &\, \epsilon \rho(x) \cos{F(x,t)} \cdot \frac{ 2 \pi}{Z} + \Big[ \epsilon \rho(x) \cos{F(x,t)} f'(x) + \epsilon \rho'(x) (1 + \sin{F(x,t)})\Big] B(x) + (\partial_x \omega - u_0')^2 \cdot \widetilde M_0 \\
    & \quad \quad +  | \partial_x \omega - u_0' | \cdot   |\omega-u_0| \cdot \widetilde M_0 + \frac{\partial H}{\partial u} (x, u_0' , u_0)  \cdot (\omega - u_0) + (\omega - u_0)^2 \cdot  \widetilde M_0 \\
    =&\,   \epsilon \rho(x) \cos{F(x,t)} \cdot \Big(\frac{ 2 \pi}{Z} + f'(x) \cdot B(x)\Big) + (\partial_x \omega - u_0')^2 \cdot  \widetilde M_0 + | \partial_x \omega - u_0' | \cdot   |\omega-u_0| \cdot \widetilde M_0  \\  &\, \quad \quad  +\epsilon \rho(x) (1 + \sin{F(x,t)}) \cdot \Big(\frac{\partial H}{\partial u} (x,   u_0' , u_0) + \frac{\rho'(x)}{\rho(x)} \cdot B(x)\Big) + (\omega - u_0)^2 \cdot  \widetilde M_0 \\
    =&\,   (\partial_x \omega - u_0')^2 \cdot  \widetilde M_0 + | \partial_x \omega - u_0' | \cdot   |\omega-u_0| \cdot \widetilde M_0 +  (\omega - u_0)^2 \cdot  \widetilde M_0+  \epsilon \rho(x) (1 + \sin{F(x,t)})\cdot \mu. 
\end{align*}
Now, let us look at each of them.
Notice that
\begin{align*}
	\partial_x \omega - u_0'=&\, \epsilon \rho'(x) (1 + \sin{F(x,t)}) + \epsilon \rho(x) \cos F(x,t) \cdot f'(x) 
	\leqslant  \epsilon \rho(x) \cdot \Big( 2\big | \frac{\rho'(x)}{\rho(x)} \big | +|f'(x) | \Big)  
	\\
	 \leqslant &\, \epsilon \rho(x) \cdot \Big(2\alpha + \frac{2\pi}{\mathcal{T} |B(x)| }  \Big), 	\end{align*}
	and
	\begin{align*}
		( \omega - u_0)^2=&\, \epsilon^2 \rho^2(x)(1 + \sin{F(x,t)})^2 \leqslant  2\epsilon^2 \rho^2(x)(1 + \sin{F(x,t)}), \\
    (\partial_x \omega - u_0 ')^2 =&\, \Big(\epsilon \rho'(x) (1 + \sin{F(x,t)}) + \epsilon \rho(x) \cos F(x,t) \cdot f'(x) \Big)^2  \\
    =&\,  \epsilon^2 \rho^2(x) (1 + \sin{F(x,t)})  \cdot \Big[ (1+\sin {F(x,t)}) \cdot \big | \frac{\rho'(x)}{\rho(x)} \big | ^2 +(1-\sin {F(x,t)}) \cdot |f'(x)|^2 \\
     &\, \quad \quad + 2\cos F(x,t) \cdot  \frac{\rho'(x)}{\rho(x)} f'(x) \Big] \\
     \leqslant &\, \epsilon^2 \rho^2(x) (1 + \sin{F(x,t)}) \cdot \Big[  2\cdot  \big | \frac{\rho'(x)}{\rho(x)} \big | ^2+2\cdot  |f'(x)|^2  + 2  \cdot \big | \frac{\rho'(x)}{\rho(x)} \big |  \cdot  |f'(x)|\Big] \\
     \leqslant &\, \epsilon^2 \rho^2(x) (1 + \sin{F(x,t)}) \cdot  \Big( 2 \alpha ^2 +\frac{8 \pi^2}{\mathcal{T}^2 |B(x)|^2} +   \frac{4 \pi \alpha }{\mathcal{T} |B(x)|}\Big),
\end{align*}
then, one deduces that
\begin{align*}
    &\,  \partial_t \omega(x,t) + H (x, \partial_t \omega(x,t) , \omega(x,t) ) \\
    \leqslant &\, \ \epsilon^2 \rho^2(x) (1 + \sin{F(x,t)}) \cdot  \Big( 2 \alpha ^2 +\frac{8 \pi^2}{\mathcal{T}^2 |B(x)|^2} +  \frac{4 \pi \alpha }{\mathcal{T} |B(x)|} +2 +2 \alpha + \frac{2\pi}{\mathcal{T}|B(x)|} \Big) \widetilde M_0  \\
   &\, \hspace{2cm} +  \epsilon \rho(x) (1 + \sin{F(x,t)})   \mu \\
     \leqslant  &\,  \epsilon \rho(x) (1 + \sin{F(x,t)}) \cdot \Big[ \mu +  \epsilon M_1 \Big( 2 \alpha ^2 +\frac{8 \pi^2}{\mathcal{T}^2 |B(x)|^2} +  \frac{4 \pi \alpha }{\mathcal{T} |B(x)|} +2 +2 \alpha + \frac{2\pi}{\mathcal{T}|B(x)|} \Big) \widetilde M_0   \Big]   \\ 
   \leqslant  &\, 0.
\end{align*}
Hence, in view of \eqref{eq:pf-thm2-step1-T-periodic}, we conclude that $\omega(x,t)$ is a $\mathcal{T}$-periodic subsolution of \eqref{eq:HJe}.
\end{proof}

In order to get the existence of periodic solutions of equation \eqref{eq:HJe}, we need to perform the following steps.
The following Lemma \ref{lem:4.2}-Lemma \ref{lem:4.4} are similar with Lemma 2.3-Lemma 2.5 in \cite{WYZ2023}.

Recall that $T^-_t\varphi(x)$  can be represented by \cite{WWY1}:
\[
T^-_t\varphi(x)=\inf_{y\in \mathbb{S} }h_{y,\varphi(y)}(x,t), \quad (x,t)\in \mathbb{S} \times(0,+\infty).
\]
 Here, the continuous functions $h_{\cdot,\cdot}(\cdot,\cdot):\mathbb{S}\times\R\times \mathbb{S} \times(0,+\infty)\to\R,\ (x_0,u_0,x,t)\mapsto h_{x_0,u_0}(x,t)$ was introduced in \cite{WWY}, called implicit action functions. Given $x_0\in \mathbb{S}$, $u_1$, $u_2\in\mathbb{R}$, 
if $u_1<u_2$, then $h_{x_0,u_1}(x,t)<h_{x_0,u_2}(x,t)$, for all $(x,t)\in \mathbb{S} \times (0,+\infty)$. See \cite{WWY,WWY2,WWY1} for more properties of implicit action functions.

\begin{lemma}\label{lem:4.2}
	For any $x_0\in \mathbb{S}$,
	$$
		\liminf_{t\to +\infty} h_{x_0,u_0(x_0)}(x_0,t)< \limsup_{t\to +\infty}  h_{x_0,u_0(x_0)}(x_0,t).
		$$
\end{lemma}
\begin{proof}
The proof is similar with \cite[Lemma 2.3]{WYZ2023}. By Proposition \ref{prop:comparison principle}, we get that
	$$
		h_{x_0,u_0(x_0)}(x,t)=h_{x_0,w(x_0,0)}(x,t)\geqslant \inf_{y\in \mathbb{S}} h_{y,w(y,0)}(x,t) = T^-_t w(x,0) \geqslant  w (x,t).
		$$
		This implies that  for any $\epsilon\in (0, \widetilde  \epsilon_1] $,
	\begin{align*}
		\limsup_{t\to +\infty} h_{x_0,u_0(x_0)}(x_0,t) \geqslant &\,   \limsup_{t\to +\infty} w (x_0,t) \geqslant \lim _{\mathbb{N} \ni k\to +\infty} w (x_0, k\mathcal{T}  + \frac{\mathcal{T}}{2}) \\
		 =&\, w (x_0, \frac{\mathcal{T}}{2})= u_0(x_0)+\epsilon \rho(x_0) +\epsilon \rho(x_0)  \sin \frac{\pi}{2}  = u_0(x_0) +2\epsilon \rho(x_0) \\
		 \geqslant &\, u_0(x_0)+ 2\epsilon M_2 > u_0(x_0).
	\end{align*}
	 where $M_2$ is as in \eqref{eq:def-M1-M2}.
	 
	Finally, \cite[(8)]{WYZ2023} shows that 
	$$
	u_0(x_0)=\liminf_{t\to +\infty} h_{x_0,u_0(x_0)}(x_0,t),	
	$$
		which completes the proof.	
		\end{proof}
		
		\medskip
		\begin{lemma}\label{lem:4.3}  \cite[Lemma 2.4]{WYZ2023} \label{lem:thm2-4}
			For any $x_0\in \mathbb{S}$, 
		$$
	u(x,t):=\lim_{n\to +\infty}h_{x_0,u_0(x_0)}(x,n\mathcal{T}+t)
	$$
	is a nontrivial $\mathcal{T}$-periodic viscosity solution of equation \eqref{eq:HJe}.
		\end{lemma}
		
		\medskip
		\begin{lemma}\label{lem:4.4} \cite[Lemma 2.5]{WYZ2023} \label{lem:thm2-5}
			 If  equation \eqref{eq:HJe} has a $\mathcal{T}$-periodic viscosity solution, then for each $n\in \mathbb{N}$, equation \eqref{eq:HJe} has infinitely many  $\frac{\mathcal{T}}{n}$-periodic viscosity solutions.
		\end{lemma}

As a direct consequence of Lemma \ref{lem:thm2-4} and Lemma \ref{lem:thm2-5},  we show that there exist infinitely many nontrivial time-periodic viscosity solutions   of equation \eqref{eq:HJs} in $  \Omega_{\delta_0} $ for    $\mu<0$, which completes the proof. \qed
 
\section{Examples}

Here are some examples of Theorem \ref{thm1} and Theorem \ref{thm2}.
 
\begin{example}
	We focus on the following contact  Hamilton-Jacobi equations \eqref{eq:HJe} with Hamiltonian
	$$
	H(x,p,u):=p^2+ p+\lambda(x)u, \quad (x,p,u)\in \mathbb{S}\times \R\times \R .
	$$
  Note that $u_0\equiv 0$ is a solution of \eqref{eq:HJs} and
$$
\mu = \int_{0}^{1} \lambda (\tau) \, d\tau, \quad \rho(x)= \exp\Big\{ \int_0^x  \mu- \lambda(\tau)  \, d \tau \Big\}.
$$
Then applying Theorem \ref{thm1} and Theorem \ref{thm2}, we have
	\begin{itemize}
	\item [(1)] If $\mu<0$, then $u_0$ is Lyapunov unstable and there exist infinitely many nontrivial time-periodic viscosity solutions of equation \eqref{eq:HJe}. 
	\item [(2)] If $\mu>0$, then $u_0$ is  asymptotic stable  and for any $\varphi\in C(\mathbb{S} ,\R )$ satisfying $ u_0-\delta_0  \leqslant \varphi\leqslant u_0+\delta_0 $,
	  $$
	\limsup_{t\to +\infty} \frac{\ln \| T_t^- \varphi-u_0 \|_\infty }{t} \leqslant  -\mu,  
	$$
	where 
	$$
	\delta_0=\frac{ \mu \cdot \Big( \displaystyle  \min_{x\in \mathbb{S}}\rho \Big)^2  }{\ 4 \ \Big( \| \rho \|_\infty +\| \rho' \|_\infty \Big) \Big(\displaystyle  \min_{x\in \mathbb{S}} \rho  + \| \rho \|_\infty +\| \rho' \|_\infty \Big)}\ .
	$$
\end{itemize} 
\end{example}

\begin{example}
	Consider the following contact  Hamilton-Jacobi equations \eqref{eq:HJe} with Hamiltonian
	$$
	H(x,p,u):=p^2+V(x) p+\lambda(x)\cdot (\sqrt{u^2+1}-\sqrt{2} ), \quad \forall (x,p,u)\in \mathbb{S}\times \R\times \R.
	$$
Here $V(x)\in C(\mathbb{S}, \R) $ satisfies $ V(x) \neq 0$ for any $x\in \mathbb{S}$. It is clear that $u_1\equiv 1$ and $u_2\equiv -1$ are two solutions of \eqref{eq:HJs} and
	$$
 \mu_1=  \frac{ \int_0^1 \frac{\sqrt{2}}{2 } \cdot \lambda(\tau)  V(\tau )^{-1} \, d\tau }{ \int_0^1   V(\tau )^{-1} \, d\tau } ,\quad \mu_2=\frac{\int_0^1 -\frac{\sqrt{2}}{2 } \cdot \lambda(\tau) V(\tau )^{-1} \, d\tau}{{ \int_0^1   V(\tau )^{-1} \, d\tau } } 
	$$
 Applying Theorem \ref{thm1} and note that   $\frac{ \int_0^1  \cdot \lambda(\tau)  V(\tau )^{-1} \, d\tau }{ \int_0^1   V(\tau )^{-1} \, d\tau }$ and $ \int_0^1 \frac{\lambda(\tau)}{|V(\tau )|} d\tau $   have the same sign, 
	\begin{itemize}
		\item[(1)]  If $ \int_0^1 \frac{\lambda(\tau)}{|V(\tau )|} d\tau >0$, then $u_1$ is local asymptotic stable and  $u_2$ is Lyapunov unstable.
        \item[(2)]  If $ \int_0^1 \frac{\lambda(\tau)}{|V(\tau )|} d\tau <0$, then $u_1$ is Lyapunov unstable and $u_2$ is local asymptotic stable.
        \item[(3)]  If $ \int_0^1 \frac{\lambda(\tau)}{|V(\tau )|} d\tau \neq 0$, then  equation \eqref{eq:HJe} admit infinitely many nontrivial time-periodic   solutions.
	\end{itemize}
	 \end{example}

\section{Appendix}\label{section:Appendix}
We provide some explanations of the statements of Theorem \ref{thm1} and Theorem \ref{thm2}.

We use $\mathcal{L}: T^*\mathbb{S}\rightarrow T\mathbb{S}$ to denote  the Legendre transform. Let
	$\bar{\mathcal{L}}:=(\mathcal{L}, Id)$, where $Id$ denotes the identity map from $\R$ to $\R$. Then
	\[
	\bar{\mathcal{L}}:T^*\mathbb{S}\times\R\to T\mathbb{S}\times\R,\quad (x,p,u)\mapsto \left(x,\frac{\partial H}{\partial p}(x,p,u),u\right)
	\]
	is a diffeomorphism. Using $\bar{\mathcal{L}}$, we can define
	the contact Lagrangian $L(x,\dot{x},u)$ associated to $H(x,p,u)$ as
	$$
	L(x,\dot x,u):=\sup_{p\in \R} \{\dot x p -H(x,p,u) \},\quad (x,\dot{x},u)\in \mathbb{S}\times\R\times\R.
	$$
	Then $L(x,\dot{x}, u)$ and $H(x,p,u)$ are Legendre transforms of each other, depending on conjugate variables $\dot{x}$ and $p$ respectively. Let $\Phi^H_t$ and $\Phi^L_t$ denote the local contact Hamiltonian flow and Lagrangian flow, respectively.
\begin{lemma}\label{lem00}\cite[Lemma 2.1, 2.2]{WYZ2023}
	Assume (H1)-(H3) and \eqref{A}. Then  $u_-:= u_0$ is of class $C^{\infty}$.
\end{lemma}
\begin{proof}
First, we claim that   equation \eqref{eq:HJs} admits a backward KAM solution $u_-$ and  a forward KAM solution $u_+$ satisfying $u_-=u_+=:u_0$.

It is clear that $u_-$ is also a	 backward weak KAM solution of  equation \eqref{eq:HJs}. Moreover, $ \displaystyle u_+:=\lim_{t\to +\infty}T^+_t u_-$ is a forward weak KAM solution.  Define the projected Aubry set by
		\[
		\mathcal{A}:=\{x\in\mathbb{S}:\ u_-(x)=u_+(x)\} \neq \emptyset.
		\]
	In view of  \cite[Theorem 1.2, Theorem 1.3]{WWY2}, $\mathcal{A}$ is nonempty. Take an arbitrary point $x\in \mathcal{A}$. Denote by $x(t)$  the global $(u_-,L,0)$-calibrated curve passing through $x$. 
		From assumption \eqref{A},	the contact Hamiltonian flow $\Phi^H_t$ has no fixed points on $\Lambda_{u_+} $. So, we get that
		\[
		\mathcal{A}=\{x(t):t\in\R\}=\mathbb{S},\quad \tilde{\mathcal{A}}=\{(x(t),\dot{x}(t),u_-(x(t))):\ t\in\R\},
		\]
		where $\tilde{\mathcal{A}}$ denotes the Aubry set in $T\mathbb{S}\times\R$  \cite[Definition 3.3]{WWY2}.   In this case, $u_-=u_+=:u_0$. 
		
		By the way,  notice that $u_-$ is semiconcave , $u_+$ is semiconvex and  
	a function is $C^{1,1}$ if and only if it is semiconcave and semiconvex  \cite[Corollary 3.3.8]{CS}, then $u_0$ is of class $C^{1,1}$. By \cite[Lemma 2.2]{WYZ2023}, $u_0$ is of class $C^{\infty}$.   The proof is complete.
\end{proof}
\begin{remark}
From the proof of Lemma \ref{lem00}, we get that 
	$\tilde{\mathcal{A}}=\{(x(t),\dot{x}(t),u_0(x(t))): t\in\R\}$.
	It is clear that $\{(x(t),\dot{x}(t),u_0(x(t))): t\in\R\}$ is an orbit of the contact Lagrangian flow $\Phi_t^L$. By Lemma \ref{lem00} again,  we also have that
	\[
	\Lambda_{u_-}=\bar{\mathcal{L}}^{-1}(\tilde{\mathcal{A}}).
	\]
	In view of assumption \eqref{A}, 
	\begin{align}\label{b1}
\dot{x}(t)=\frac{\partial H}{\partial p}\Big(\bar{\mathcal{L}}^{-1}\big(x(t),\dot{x}(t),u_0(x(t))\big)\Big)\neq 0,\quad \forall t\in\mathbb{R}.
	\end{align}
	By \eqref{b1} and the structure of the unit circle $\mathbb{S}$, one can deduce that $x(t)$ is $\mathcal{T}$-periodic for some $\mathcal{T}>0$ and thus $\tilde{\mathcal{A}}=\{(x(t),\dot{x}(t),u_0(x(t))): t\in\R\}$ is an orbit of the contact Lagrangian flow $\Phi_t^L$.
	
	Let  
$
		(x(t),p(t),u_0(x(t))):={\bar{\mathcal{L}}}^{-1}((x(t),\dot{x}(t),u_0(x(t))),
$
	then 
	\[
	\Lambda_{u_-}=\{(x(t),p(t),u_0(x(t))):\ t\in\mathbb{R}\}. 
	\]
 By  \eqref{b1}, one follows that
	$$
	\mathcal{T}=\big|\int_0^1  \big( B(\tau )\big)^{-1} d\tau \big| \in \R^+.
	$$
\end{remark}

	\begin{lemma}\label{lem1111}
	Define $\mu$ as in \eqref{eq:def-mu}, then
			$$
		\mu =\frac{1}{\mathcal{T} } \int_0^\mathcal{T} \frac{\partial H}{\partial u} ( x(s),d_x u_0(x(s)) , u_0(x(s)) ) ds  . 
		$$
	\end{lemma}
	\begin{proof}
	Due to Lemma \ref{lem2.2}, $u_0$ is of class $C^\infty$.  By  assumption \eqref{A}  , we get 
	$$
	 \frac{\partial H}{\partial p} (x, \,  u_0'(x),u_0(x) )  >0, \quad \forall x\in \mathbb{S}, \quad \text{or} \quad  \frac{\partial H}{\partial p} (x, \,  u_0'(x),u_0(x) )  <0, \quad \forall x\in \mathbb{S} .
	$$
	If  $\frac{\partial H}{\partial p} (x, \,  u_0'(x),u_0(x) )  >0$ holds for any $   x\in \mathbb{S}$, then 
		\begin{align*}
&\,    \frac{1}{ \mathcal{T} } \, \int_0^{\mathcal{T}} \frac{\partial H}{\partial u} (x(s), \,  u_0'(x( s)),u_0(x( s) ) ) \ d  s \\ = &\,    \frac{1}{ \mathcal{T} } \, \int_0^{\mathcal{T}} \frac{ \frac{\partial H}{\partial u} (x(s), \,  u_0'(x(s)),u_0(x(s)) )}{\quad  \dot x(s)  \quad } \  d \,   x(s) \\ 
 =  &\, \frac{1}{ \mathcal{T} } \int_0^1 \frac{ \frac{\partial H}{\partial u} (x, \,  u_0'(x),u_0(x) )}{\quad \frac{\partial H}{\partial p} (x, \,  u_0'(x),u_0(x) )  \quad } \ d  x  \\
 =&\,\frac{ \int_0^1   \frac{\partial H}{\partial u} (x, \,  u_0'(x),u_0(x) ) \cdot \big( B(x )\big)^{-1}  \ d  x}{\int_0^1 \big( B(x )\big)^{-1} d x }   =\mu .
\end{align*}
 It is similar for the case that $\frac{\partial H}{\partial p} (x, \,  u_0'(x),u_0(x) )  <0$ holds for any $   x\in \mathbb{S}$. This completes the proof.
	\end{proof}
	
	  \bigskip

\noindent   {\bf There is no conflict of interest and no data in this paper.}
  
  	\medskip

	  \section*{Acknowledgements}
Jun Yan is supported by NSFC Grant Nos. 12171096,  12231010. Kai Zhao is supported by NSFC Grant No. 12301233.


%


\end{document}